\documentclass[a4paper,12pt]{article}
\usepackage[english]{babel} 
\usepackage{enumerate}
\usepackage{hyperref}
\usepackage{amsmath,amsfonts,amsthm}
\usepackage{tikz}
\allowdisplaybreaks
\usepackage{mathtools}
\usepackage{pdfsync} 
\usepackage{graphicx,float}
\usepackage{mathrsfs} 
\usepackage{array} 

\graphicspath{{Figures/}}
\usepackage[utf8]{inputenc}
\usepackage{authblk}
\usepackage[T1]{fontenc}
\usepackage{lmodern}

\usepackage{color}
\newtheorem{lemma}{Lemma}

\newtheorem{coro}{Corollary}
\newtheorem{prop}{Proposition}

\newtheorem{ex}{Example}
\newtheorem{test}{Test}
\newtheorem{defn}{Definition}

\newcommand{\e}{\mathrm{e}} 
\newcommand{\E}{\mathbb{E}} 
\newcommand{\Esp}[1]{\Esp\left(#1\right)} 
\renewcommand{\P}{\mathbb{P}} 
\newcommand{\C}{\mathbb{C}} 
\newcommand{\R}{\mathbb{R}} 
\newcommand{\N}{\mathbb{N}_0} 
\newcommand{\Npos}{\mathbb{N}_+} 
\newcommand{\Z}{\mathbb{Z}} 
\newcommand{\Lp}{L} 
\newcommand{\PGF}{\mathcal{G}} 
\newcommand{\LT}[1]{\mathcal{L}\{#1\}} 
\newcommand{\LTsub}[1]{\mathcal{L}_{#1}} 

\newcommand{\dd}{\,\mathrm{d}} 
\newcommand{\Oh}{{\mathcal{O}}} 
\newcommand{\defeq}{\equiv}
\newcommand{\eqdistr}{\stackrel{{\scriptstyle \mathcal{D}}}{=}}
\newcommand{\iidDist}{\overset{\mathrm{i.i.d.}}{\sim}}

\usepackage{xspace}
\newcommand{\iid}{\textbf{iid}\xspace}
\newcommand{\rv}{\textbf{rv}\xspace}
\newcommand{\rvs}{\textbf{rv's}\xspace}
\newcommand{\pdf}{\textbf{pdf}\xspace}

\newcommand{\pmf}{\textbf{pmf}\xspace}
\newcommand{\cdf}{\textbf{cdf}\xspace}
\newcommand{\svf}{\textbf{sf}\xspace}
\newcommand{\slp}{\textbf{slp}\xspace}
\newcommand{\slps}{\textbf{slp's}\xspace}
\newcommand{\cf}{\textbf{cf.}\xspace}
\newcommand{\VaR}{\textbf{VaR}\xspace}

\newcommand{\eg}{\textbf{e.g.}\xspace}
\newcommand{\ie}{\textbf{i.e.}\xspace}
\newcommand{\etal}{\textbf{et al.}\xspace}
\makeatletter
\def\timenow{\@tempcnta\time
\@tempcntb\@tempcnta
\divide\@tempcntb60
\ifnum10>\@tempcntb0\fi\number\@tempcntb
:\multiply\@tempcntb60
\advance\@tempcnta-\@tempcntb
\ifnum10>\@tempcnta0\fi\number\@tempcnta}
\makeatother

\usepackage{fancyhdr}
\begin{document}
\title{Orthogonal polynomial expansions to evaluate stop-loss premiums}
\author{Pierre-Olivier Goffard\footnote{pierre-olivier.goffard@univ-lyon1.fr} }
\author{Patrick J.\ Laub\footnote{patrick.laub@univ-lyon1.fr. Research conducted within the DAMI – Data Analytics and Models for Insurance - Chair under the aegis of the Fondation du Risque, a joint initiative by UCBL and BNP Paribas Cardif.}}
\affil{\footnotesize Univ Lyon, Universit{\'e} Lyon 1, LSAF EA 2429 , Institut de Science Financi{\`e}re et d'Assurances, 50 Avenue Tony Garnier, F-69007 Lyon, France}
\date{\today}
\maketitle
\vspace{3mm}
\begin{abstract}
A numerical method is proposed to evaluate the survival function of a compound distribution and the stop-loss premiums associated with a non-proportional global reinsurance treaty. The method relies on a representation of the probability density function in terms of Laguerre polynomials and the gamma density. We compare the method against a well established Laplace transform inversion technique at the end of the paper.
\end{abstract}

\section{Introduction}\label{sec:Introduction}

Consider the random variable (\rv)
\begin{equation*}\label{eq:AggregatedClaimAmountsRV}
S_N=\sum_{k=1}^{N}U_{k},
\end{equation*}
where $N$ is a counting \rv and $\{U_{k}\}_{k\in\Npos}$ is a sequence of \rvs which are independent and identically distributed (\iid), non-negative, and independent of $N$.
We denote the probability density function (\pdf) of $S_N$ as $f_{S_N}$, and its survival function (\svf) as
\begin{equation*}\label{eq:DefinitionSurvivalFunction}
\overline{F}_{S_N}(x)= \P(S_N>x),\quad \text{ for } x\geq0.
\end{equation*}
This paper concerns approximations of $f_{S_N}$ and $\overline{F}_{S_N}$, though we begin with a discussion of how $S_N$ is used in actuarial science.

Frequently, $S_N$ models the aggregated losses of a non-life insurance portfolio over a given period of time---here $N$ represents the number of claims and $U_k$ the claim sizes---yet other applications also exist. Actuaries and risk managers typically want to quantify the risk of large losses by a single comprehensible number, a risk measure.

One popular risk measure is the Value-at-Risk (\VaR). In actuarial contexts, the \VaR at level $\alpha \in (0,1)$ is defined such that the probability of (aggregated) losses exceeding the level $\VaR$ is at most $1-\alpha$. We denote this $\alpha$-quantile as
\begin{equation*}
\text{VaR}_{S_N}(\alpha)=\inf\{x\geq0, F_{S_N}(x)\geq \alpha\}.
\end{equation*}
Following the European recommendation of the Solvency II directive, the standard value for $\alpha$ is $0.995$, see \cite{EIOPA}. It is used by risk managers in banks, insurance companies, and other financial institutions to allocate risk reserves and to determine solvency margins. Also, we have stop-loss premiums (\slps) which are risk measures that are commonly used in reinsurance agreements.

A reinsurance agreement is a common risk management contract between insurance companies, one called the ``cedant'' and the other the ``reinsurer''. Its aim is to keep the cedant's long-term earnings stable by protecting the cedant against large losses. The reinsurer absorbs part of the cedant's loss, say $f(S_N)$ where $0\leq f(S_N)\leq S_N$, leaving the cedant with $I_{f}(S_N)=S_N-f(S_N)$. In return, the cedant pays a premium linked to
\begin{equation*}\label{eq:DefinitionReinsurancePremium}
\Pi=\E[f(S_N)],
\end{equation*}
under the expected value premium principle.

In practice, there are a variety of reinsurance designs from which an insurer can choose. We focus in this work on the stop-loss reinsurance treaty associated with the following ceded loss function
\begin{equation*}\label{eq:StopLossCededFunction}
f(S_N)=(S_N-a)_{+},\text{ }a\geq0,
\end{equation*}
where $a$ is referred to as the retention level or priority. The ratemaking of the stop-loss reinsurance policy requires the evaluation of
\begin{equation}\label{eq:DefUsualStopLossPremium}
\Pi_{a}(S_N)=\E\left[(S_N-a)_{+}\right],
\end{equation}
also known as the usual stop loss premium (\slp).

One variation is the limited stop-loss function,
\begin{equation}\label{eq:LimitedStopLossCededFunction}
f(S_N)=\min[(S_N-a)_{+},b],\text{ }b\geq0,
\end{equation}
where $b$ is called the limit. The limited stop-loss function \eqref{eq:LimitedStopLossCededFunction} is very appealing in practice because it prevents the cedant from over-estimating their losses and therefore over-charging the reinsurer.
Also, the change-loss function is defined as
\begin{equation*}\label{eq:ChangeLossCededFunction}
f(S_N)=c(S_N-a)_{+},\text{ }0\leq c\leq1,
\end{equation*}
which is in between stop-loss and quota-share reinsurance. The ratemaking in each case requires the expectation in \eqref{eq:DefUsualStopLossPremium}.

From a practical point of view, a reinsurance treaty over the whole portfolio is less expensive to handle than one which involves claim-by-claim management. It also grants protection in the event of an unusual number of claims, triggered for instance by a natural disaster. From a theoretical point of view, it is well known that the stop-loss ceded function allows one to minimize the variance of the retained loss for a given premium level, see for instance the monograph of Denuit \etal \cite{DeDhGoKa06}. Recently, it has been shown that stop-loss reinsurance is also optimal when trying to minimize the \VaR and the expected shortfall of the retained loss, see the works of Cai \etal \cite{CaTaWeZh08}, Cheung \cite{Ch10}, and Chi and Tan \cite{ChTa11}. Note that some other ceded loss functions appear in their work, they are however very close to the stop-loss one.

Unfortunately, one is seriously constrained when calculating these quantities analytically, as there are only a few cases where either the \pdf or the \svf is available in a simple tractable form. To compute the \VaR or \slp we must find fast and accurate approximations for these functions.

We discuss the use of an approximation of the \pdf in terms of the gamma density and its orthonormal polynomials. This method has been studied in the recent works of Goffard \etal \cite{GoLoPo15} and Jin \etal \cite{JiPrRe16}, though it goes back to Bowers \cite{bowers1966expansion} at least. We emphasize here the computational aspect of this numerical method and detail some practical improvements. An exponential change of measure can be used to recover the \pdf of $S_N$ when the claim sizes are governed by a heavy-tailed distribution. Introducing an exponential change of measure has been successfully applied in many previous works, \eg, in the work of Asmussen \etal \cite{asmussen2016orthonormal} to recover the density of the sum of lognormally distributed random variables and earlier by Mnatsakanov \etal \cite{Mnatsakanov2008} to calculate ruin probabilities.

This method is compared to a numerical inversion of the Laplace transform which is known to be efficient to recover the survival function of a compound distribution. The critical step in Laplace inversion is to select which numerical integration technique to apply. We implement a method inspired by the work of Abate and Whitt \cite{Abate1992} which is very similar to the method of Rolski \etal \cite[Chapter 5, Section 5]{RoScScTe08}. An approximation of the \slp is then proposed relying on the connection with the survival function of the equilibrium distribution of $S_N$. Note that Dufresne \etal \cite{DuGaMo09} successfully applied a Laplace inversion based technique to the evaluation of \slp. We chose not to include Panjer's algorithm or the Fast Fourier Transform methodology in the comparison study because they both require the discretization of the claim sizes distribution and also because they have been already compared in the work of Embrechts and Frei \cite{embrechts2009panjer}.

To close this section, we want to emphasize the fact that the numerical methods also apply in a risk theory framework. The infinite-time ruin probability in the compound Poisson ruin model is equal to the survival function of a compound geometric distribution. The polynomial approximation and the Laplace inversion methods have been employed, and compared to solve this particular problem in the work of Goffard \etal \cite{GoLoPo16}. We add a more original application by noting that the finite-time non-ruin probability with no initial reserves, again under the classical risk model assumptions, may be rewritten as the \slp associated with a compound Poisson distribution where the priority is expressed in terms of the premium rate and the time horizon.

The rest of the paper is organized as follows. Section \ref{sec:Preliminaries} introduces compound distributions, and details their role in risk theory. Section \ref{sec:PolynomialApproximation} presents the approximation method based on orthogonal polynomials. Section \ref{sec:NumericalInversionLaplaceTransform} presents the approximation through the numerical inversion of the Laplace transform. Section \ref{sec:NumericalIllustrations} is devoted to numerical illustrations where the performances of the two methods are compared; the \textsc{Mathematica} code used in this section is available online \cite{StoplossCode}.

\section{Compound distributions and risk theory}\label{sec:Preliminaries}

After setting up some notational conventions for Laplace transforms, see Definition \ref{def:TransformDefs}, compound distributions are introduced along with a brief account of their importance in risk modeling.
\begin{defn} \label{def:TransformDefs} For a function $f : \R_+ \to \R_+$, we define
\[
	\LT{f}(t) \defeq \int_{0}^{\infty} \e^{-t x} f(x) \dd x\,, \quad \text{for } t \in \C \text{ with } \Re(t) \ge 0 \,, \\
\]
to be the corresponding Laplace transform.
For a positive random variable $X$ with \pdf $f_X$, we write $\LTsub{X}(t) \defeq \LT{f_X}(t) = \E\,\e^{-tX}$. \hfill $\diamond$
\end{defn}

Note that
\[ \LT{ F_X }(t) = \frac{\LT{f_X}(t)}{t} = \frac{\LTsub{X}(t)}{t} \,, \text{ and } \]
\[ \LT{ \overline{F}_X }(t) = \frac1t - \LT{F_X(x) }(t) = \frac{1 - \LTsub{X}(t)}{t}  \,, \text{ for } t>0.
\]

\subsection{Compound distribution}
Let $S_N=\sum_{k=1}^{N}U_k$ be the aggregated claim amounts associated with a non-life insurance portfolio over a fixed time period. The number of claims, also called the claim frequency, is modeled by a counting random variable $N$ having a probability mass function $f_N$. The claim sizes form a sequence $\{U_k\}_{k\in \Npos}$ of \iid non-negative random variables with common \pdf $f_{U}$. We further assume that the claim sizes are independent from the claim frequency.

As $S_N=0$ whenever $N=0$ (assuming this occurs with positive probability), the distribution of $S_N$ is the sum of a singular part (the probability mass $\P(S_N=0)=f_N(0)>0$) and a continuous part (describing $S_N$ where $N>0$) with a defective \pdf $f_{S_N}^+$ and \cdf $F_{S_N}^+$. From the law of total probability, we have
\begin{equation}\label{eq:DefectivePDFCompoundDistribution}
f_{S_N}^+(x)=\sum_{n=1}^{\infty}f_N(n)f_{U}^{\ast n}(x),~x\geq0.
\end{equation}

This density is intractable because of the infinite series. Furthermore, the summands are defined by repeated convolution of $f_{U}$ with itself which are rarely straightforward to evaluate. The methods presented in this work rely on the knowledge of the Laplace transform of $S_N$, given by
\begin{equation*}\label{eq:TransformsForCompoundDistribution}
\LTsub{S_N}(t) = \PGF_N[ \LTsub{U}(t) ] \,,
\end{equation*}
where $\PGF_N(t) \defeq \E (t^N)$ is the probability generating function of $N$. The simple expression of the Laplace transform has made possible the use of numerical methods involving the moments or transform inversion to recover the distribution of $S_N$. The distribution is typically recovered using Panjer's algorithm or a Fast Fourier Transform algorithm based on the inversion of the discrete Fourier transform; these two methods are compared in the work of Embrechts and Frei \cite{embrechts2009panjer}. Our orthogonal polynomial method involves the standard integer moment sequence for $S_N$, in contrast to more exotic types of moments used by some recent methods. Gzyl and Tagliani \cite{GzTa12} uses the fractional moments within a max-entropic based method, while Mnatsakanov and Sarkisian \cite{Mnatsakanov2013} performs an inversion of the scaled Laplace transform via the exponential moments. In addition to proposing an approximation for the survival function of $S_N$, we provide an efficient way to compute the usual \slp \eqref{eq:DefUsualStopLossPremium} for reinsurance applications.

\subsection{Risk theory}
In the classical risk model, the financial reserves of a non-life insurance company are modeled by the risk reserve process $\{R(t),t\geq0\}$, defined as
\begin{equation*}\label{eq:RiskReserveProcess}
R(t)=u+ct-\sum_{k=1}^{N(t)}U_k.
\end{equation*}
The insurance company holds an initial capital of amount $R(0)=u\geq0$, and collects premiums at a constant rate of $c>0$ per unit of time. The number of claims up to time $t\geq0$ is governed by a homogeneous Poisson process $\{N(t),t\geq0\}$ with intensity $\lambda$. The claim sizes are \iid non-negative random variables independent from $N(t)$.

One of the goals of risk theory is to evaluate an insurer's ruin probability, that is, the probability that the financial reserves eventually fall below zero. Of interest are both the finite-time ruin probability $\psi(u,T)$ and the infinite-time ruin probability, also called the \textit{probability of ultimate ruin}, $\psi(u)$, which are defined as
\begin{equation*}\label{eq:FiniteTimeRuinProbability}
\psi(u,T)=\P\Big( \inf_{0 \le t \le T} R(t) \leq 0 \Big),
\end{equation*}
and
\begin{equation*}\label{eq:InfiniteTimeRuinProbability}
\psi(u)=\P\Big( \inf_{t \geq 0}\,R(t) \leq 0 \Big).
\end{equation*}
For a general background on risk theory and the evaluation of ruin probabilities, we refer the reader to the monograph of Asmussen and Albrecher \cite{asmussen2010ruin}.

The first connection between compound distributions and ruin probabilities is the following.
If the net benefit condition is satisfied, \ie if the premium rate exceeds the average cost of aggregated claims per unit of time, then the infinite-time ruin probability is given by the survival function of a geometric compound distribution. More precisely,
\begin{equation*}\label{eq:PollaczeckKhinchineFormula}
\psi(u) = \P\left(S_N \defeq \sum_{k=1}^{N}U^\ast_{k}>u\right)
= (1 - \rho) \sum_{n=1}^\infty \rho^n \overline{F}_{U^\ast}^{\ast n}(u),
\end{equation*}
with $N \sim \mathsf{Geom}_0(\rho)$, $\rho = \lambda\E(U)/c<1$, and with \iid $U^\ast_{k}$ with \pdf $f_{U^\ast}(x)= \overline{F}_U(x) / \E(U)$. This result is known as the Pollaczeck--Khinchine formula, see for instance Asmussen and Albrecher \cite[Chapter IV, (2.2)]{asmussen2010ruin}. Thus it is possible to evaluate the infinite-time ruin probability via Panjer's algorithm. If we are able to determine the Laplace transform of $S_N$ then we can also apply the polynomial method of Goffard \etal \cite{GoLoPo15}, the fractional moment based method of Gzyl \etal \cite{GzNITa13}, and the exponential moments based method of Mnatsakanov \etal \cite{Mnatsakanov2015}.

The second connection links the finite-time ruin probability with no initial reserves to the \slp associated with a compound distribution.
If $N(t) \sim \mathsf{Poisson}(\lambda t)$ (\ie claims arrive as a homogeneous Poisson process) then the finite-time ruin probability is given by
\begin{align} \label{CramersFormula}
\psi(0, T)
&= 1 - \frac{1}{cT} \int_0^{cT} \P \left( \sum_{i=1}^{N(T)} U_i \le x \right) \dd x \,.
\end{align}
This implies $\psi(0,T) = \E [ \min\{S_{N(T)}, cT \} ]/cT$ where $S_{N(T)} \defeq \sum_{i=1}^{N(T)} U_i$, and hence
\begin{align} \label{eq:ConnectionFiniteTimeRuinProbabilityStopLossPremium}
\psi(0, T)
&= (cT)^{-1} \Big[  \E[N(T)] \, \E[U_1] - \Pi_{cT}(S_{N(T)}) \Big] \,.
\end{align}
Lef\`evre and Picard \cite[Corollary 4.3]{LePi11} show that equations \eqref{CramersFormula} and \eqref{eq:ConnectionFiniteTimeRuinProbabilityStopLossPremium} hold in the more general case where the claim arrival process forms a \textit{mixed Poisson process}. This connection has been exploited recently in Lef\`evre \etal \cite{LeTrZu17} where the influence of the claim size distribution on the ruin probabilities is studied via stochastic ordering considerations.

\section{Orthogonal polynomial approximations}\label{sec:PolynomialApproximation}

Orthogonal polynomials have been used at multiple occasions in applied probability and statistics, for instance in the study of stochastic processes in the textbook of Schoutens \cite{Sc12} and in the derivation of summation formulas in Diaconis and Zabell \cite{DiZa91}. In this work, we use them to derive an approximation formula to recover an unknown probability measure from the knowledge of its moments.

\subsection{Approximating general density functions} \label{ssec:PolynomialApproxIntro}
 Let $X$ be an arbitrary random variable with \pdf $f_X$ with respect to some measure $\lambda$ (typically Lebesgue measure on an interval or counting measure on a subset of $\Z$). We assume that the density is unknown and we propose an approximation of the form
\begin{equation}\label{eq:PolynomialApproximation}
\widehat{f}_{X}(x)=\sum_{k=0}^{K}q_{k}Q_{k}(x)f_{\nu}(x),
\end{equation}
where $f_{\nu}$ is the reference or basis density, associated to a probability measure $\nu$ absolutely continuous with respect $\lambda$. The sequence $\{Q_{k},k \geq 0\}$ is made of polynomials, orthonormal with respect to $\nu$ in the sense that
\begin{equation*}\label{eq:OrthonormalProperty}
\left<Q_{k},Q_{l}\right>_\nu =\int Q_{k}(x)Q_{l}(x)\dd\nu(x)= \delta_{k,l},\text{ }k,l\in\N.
\end{equation*}
This sequence is generated by the Gram--Schmidt orthogonalization procedure 
which is only possible if $\nu$ admits moments of all orders. If additionally there exists $s>0$ such that
\begin{equation}
\int \e^{s |x|}\dd\nu(x)<\infty,\nonumber
\end{equation}
then the sequence of polynomials $\{Q_{k},k \geq 0\}$ forms an orthonormal basis of $\Lp^2(\nu)$ which is the space of all square integrable functions with respect to $\nu$, see the monograph by Nagy \cite[Chapter 7]{Na65}. Therefore, if $f_{X}/f_{\nu}\in \Lp^2(\nu)$ then the polynomial representation of the density of $X$ with respect to $\nu$ follows from orthogonal projection, namely we have
\begin{equation}\label{eq:PolynomialRepresentationTheory}
f_{X}(x)/f_{\nu}(x)=\sum_{k=0}^{\infty}\left<f_{X}/f_{\nu} , Q_{k}\right>_\nu Q_{k}(x).
\end{equation}
We label the coefficients of the expansion as $\{q_{k},k \geq 0\}$, noting that
\begin{equation}
q_k\defeq \left<f_{X}/f_{\nu} , Q_{k}\right>_\nu =\int Q_{k}(x)f_{X}(x)\frac{\dd \nu(x)}{f_\nu(x)} =\E\left[Q_{k}(X)\right], \text{ }k\in \N,\nonumber
\end{equation}
and thus we can rearrange \eqref{eq:PolynomialRepresentationTheory} to be
\begin{equation}\label{eq:PolynomialRepresentation}
f_{X}(x)=\sum_{k=0}^{\infty} q_k Q_{k}(x) f_{\nu}(x).
\end{equation}
The approximation \eqref{eq:PolynomialApproximation} follows by simply truncating the series to $K+1$ terms.

The Parseval relationship, $\sum_{k=1}^{\infty}q_{k}^{2}=||f_{X}/f_{\nu}||_\nu^{2}$, ensures that the sequence $\{q_{k},k \geq 0\}$ tends toward $0$ as $k$ tends to infinity. The accuracy of the approximation \eqref{eq:PolynomialApproximation}, for a given order of truncation $K$, depends on how swiftly these coefficients decay. The $\Lp^2$ loss associated with the approximation of $f_{X}/f_{\nu}$ is $\sum_{k=K+1}^{\infty}q_{k}^{2}$.

Typical choices of reference distributions are ones that belong to a Natural Exponential Family with Quadratic Variance Function (NEF-QVF) which includes the normal, gamma, hyperbolic, Poisson, binomial, and Pascal distributions. This family of distributions is convenient as the associated orthogonal polynomials are classical, see the characterization by Morris \cite{Mo82} and see also the extension by Letac and Mora \cite{LeMo90} to the case of Natural Exponential families with Cubic Variance Function. The polynomials are known explicitly, thus we avoid the time-consuming Gram--Schmidt orthogonalization procedure. Furthermore, it has been shown in a paper by Provost \cite{Pr05} that the recovery of unknown densities from the knowledge of the moments of the distribution naturally leads to approximation in terms of the gamma density and Laguerre polynomials when $X$ admits $\R_{+}$ as support, and in terms of the normal density and Hermite polynomials when $X$ has $\R$ as support.

\subsection{Approximating densities of positive random variables} \label{ssec:PolynomialApproxPositive}

To approximate the \pdf for positive $X$, a natural candidate for the reference density is the gamma density. It has been proven to be efficient in practice, see the work of Goffard \etal \cite{GoLoPo15,GoLoPo16}, and Jin \etal \cite{JiPrRe16}. The work of Papush \etal \cite{PaPaPo01} showed that among the gamma, normal and lognormal distributions, the gamma distribution seems to be better suited to model certain aggregate losses. The lognormal distribution is a problematic choice. Even though the orthogonal polynomials are available in a closed form (c.f.\ Asmussen \etal \cite{asmussen2016orthonormal}) they do not provide a complete orthogonal system of the $\Lp^2$ space. The case of the inverse Gaussian as basis received a treatment in the work of Nishii \cite{Ni96}, where it is shown that the only way to get a complete system of polynomials is by using the Gram--Schmidt orthogonalization procedure. Differentiating the density (as it is done in the case of NEF-QVF) does not lead to an orthogonal polynomial system, and starting from the Laguerre polynomials leads to a system of orthogonal functions which is not complete. A solution might be to exploit the bi-orthogonality property pointed out in the work of Hassairi and Zarai \cite{HaZa04}. To close this review of reference densities, we mention the work of Nadarajah \etal \cite{NaChJi16} where Weibull and exponentiated exponential distributions are considered as reference density.

The $\mathsf{Gamma}(r,m)$ distribution, where $r$ is the shape parameter and $m$ is the scale parameter, has a \pdf
\begin{equation*}\label{eq:PDFGamma}
f_\nu(x)\defeq\gamma(r,m,x)=\frac{x^{r-1}\e^{-\frac{x}{m}}}{\Gamma(r)m^{r}},\text{ }x\in\R^{+},
\end{equation*}
where $\Gamma(\cdot)$ denotes the gamma function.\footnote{For the distributions in this paper, we use \textsc{Mathematica}'s parametrization, \eg the exponential and Erlang distributions are $\mathsf{Exp}(\lambda) = \mathsf{Gamma}(1, 1/\lambda)$ and $\mathsf{Erlang}(n, m) = \mathsf{Gamma}(n, 1/m)$.}
The associated orthonormal polynomials are given by
\begin{equation*}\label{eq:GeneralizedLaguerrePolynomials}
Q_{n}(x)
=(-1)^{n} \binom{n + r - 1}{n}^{-\frac12} L_{n}^{r-1}\big(\frac{x}{m}\big)
=(-1)^{n} \left( \frac{\Gamma(n+r)}{\Gamma(n+1)\Gamma(r)} \right)^{-\frac12} L_{n}^{r-1}\big(\frac{x}{m}\big),
\end{equation*}
where $\{L_{n}^{r-1},n\geq0\}$ are the generalized Laguerre polynomials,
\begin{equation*}\label{eq:GeneralizedLaguerrePolynomialsExpression}
L_{n}^{r-1}(x)
=\sum_{i=0}^{n} \binom{n + r - 1}{n - i} \frac{(-x)^i}{i!}
=\sum_{i=0}^{n} \frac{\Gamma(n+r)}{\Gamma(n-i+1)\Gamma(r+i)}\frac{(-x)^i}{i!}, \text{ }n\geq 0,
\end{equation*}
\cf the classical book by Szeg{\"o} \cite{Szegoe1939}.

\begin{lemma}\label{Lemma:PseudoMixtureOfGammas}
If $\nu$ is $\mathsf{Gamma}(r,m)$ and $f_{X}/f_{\nu}\in \Lp^2(\nu)$, then the polynomial expansion \eqref{eq:PolynomialRepresentation} may be rewritten as
\begin{align}\label{eq:PseudoMixtureOfGammas}
f_{X}(x)&= \sum_{i=0}^{\infty} p_i \gamma(r+i,m,x),
\end{align}
where
\begin{equation}\label{eq:PiExpression}
p_i
=\sum_{k=i}^\infty q_k \frac{(-1)^{i+k}}{i! \, (k-i)!} \sqrt{\frac{k! \Gamma(k+r)}{\Gamma(r)}},
\end{equation}
and the function $\gamma(r,m,x)$ is the \pdf of the $\mathsf{Gamma}(r,m)$ distribution.
\end{lemma}
\begin{proof}
If we change the sum in \eqref{eq:PolynomialRepresentation} from iterating over Laguerre polynomials to iterating over monomials we get
\[
f_{X}(x)=\sum_{k=0}^{\infty} q_k Q_{k}(x) \gamma(r,m,x) = \sum_{i=0}^\infty c_i x^i \gamma(r,m,x) \, ,
\]
where
\begin{align*}
c_i
&= \sum_{k=0}^\infty \text{Coefficient}(x^i, q_k  Q_k(x))
= \frac{(-1)^i}{m^i i!} \sum_{k=i}^\infty q_k (-1)^{k} \binom{k + r - 1}{k}^{-\frac12} \binom{k + r - 1}{k - i} \,.
\end{align*}
We also note that
\[ x^i \gamma(r, m, x) = m^i \frac{\Gamma(r+i)}{\Gamma(r)} \gamma(r+i, m, x), \]
so
\[ f_{X}(x) = \sum_{i=0}^\infty c_i m^i \frac{\Gamma(r+i)}{\Gamma(r)} \gamma(r+i, m, x) =  \sum_{i=0}^\infty p_i \gamma(r+i, m, x),  \]
where we have set $p_i = c_i m^i \Gamma(r+i) / \Gamma(r)$.
\end{proof}

A sufficient condition for $f_{X}/f_{\nu}\in\Lp^2(\nu)$ is
\begin{equation}\label{eq:IntegrabilityCondition}
f_{X}(x)=
\begin{cases}
\Oh(\e^{- x/\delta})&\text{ as }x\rightarrow\infty\text{ with }m>\delta/2,\\
\Oh(x^{\beta})&\text{  as }x\rightarrow0\text{ with } r < 2( \beta + 1).
\end{cases}
\end{equation}
When $X$ has a well-defined moment generating function one can typically choose $r$ and $m$ so this integrability condition is satisfied. To be specific, define the radius of convergence of a random variable $X$ as
$$
\rho_X \defeq \sup\{s>0, ~ \LTsub{X}(-s)<\infty\},
$$
and consider the following result.

\begin{prop}\label{prop:ExpBound}
Say that $X$ is a \rv whose radius of convergence $\rho_X$ exists and whose density $f_X(x)$ is decreasing for $x>a$, then
\begin{equation}\label{eq:GrandOExp}
f_X(x) = \Oh(\exp\{- x \rho_X\}), \quad \text{as } x \to \infty \,.
\end{equation}
\end{prop}
\begin{proof}
Let $x>a$ and $s\in(0,\rho_{S_N})$ then
\begin{align*}
\LTsub{X}(-s)\
&\ge\ \int_{a}^{x}\e^{sy}f_X(y)\text{d}y\\
&=\ \frac{1}{s}\left[f_X(x)\e^{sx}-f_X(a)\e^{sa}\right]-\frac{1}{s}\int_{a}^{x}\e^{sy}f_X(y)\text{d}y\\
&\ge\ \frac{1}{s}\left[f_X(x)\e^{sx}-f_X(a)\e^{sa}\right].
\end{align*}
The \pdf $f_X$ is bounded from above with
$$
f_X(x) \le \left[s\LTsub{X}(-s)+f_X(a)\e^{sa}\right]\e^{-sx},
$$
which is equivalent to \eqref{eq:GrandOExp} when taking $s \nearrow \rho_X$.
\end{proof}

Proposition~\ref{prop:ExpBound} implies that for the \rvs whose densities are eventually decreasing, the first integrability condition (concerning the right tail) is satisfied if $m>1/(2 \rho_X)$.

When we consider heavy-tailed distributions, which is a desirable model characteristic in the applications, the integrability condition cannot be satisfied. The work-around is to use the expansion
\[f_{\theta}(x)\defeq \frac{\e^{-\theta x}}{\LTsub{X}(\theta)} f_{X}(x)
= \sum_{k=0}^\infty q_k Q_k(x) f_\nu(x), \]
for some $\theta > 0$. Thus, we can use
\begin{eqnarray}
f_{X}(x)
&=& \e^{\theta x} \LTsub{X}(\theta) \sum_{k=0}^\infty q_k Q_k(x) f_\nu(x) \nonumber\\
&=& \e^{\theta x} \LTsub{X}(\theta) \sum_{i=0}^{\infty} p_i \gamma(r+i,m,x) \label{eq:polynomial_expansion_exponentially_tilted_pdf}
\end{eqnarray}
and since, when $1-m \theta > 0$,
\[ \e^{\theta x} \gamma(r+i, m, x) = (1-m \theta)^{-(r+i)} \gamma\big(r+i, \frac{m}{1-m\theta}, x \big)  \]
we have
\begin{align*}
f_{X}(x)
= \LTsub{X}(\theta) \sum_{i=0}^{\infty} p_i (1-m \theta)^{-(r+i)} \gamma\big(r+i, \frac{m}{1-m\theta}, x \big)
= \sum_{i=0}^{\infty} \widetilde{p}_i \gamma\big(r+i, \widetilde{m}, x \big),
\end{align*}
where
\[ \widetilde{p}_i = \frac{\LTsub{X}(\theta) p_i }{(1-m \theta)^{r+i}}
\quad\text{and}\quad
\widetilde{m} = \frac{m}{1-m\theta} \,.
\]
Calculating the $q_i$'s and $p_i$'s, a topic covered in Section~\ref{sssec:ComputingOrthogonalCoefficients}, requires a Laplace transform of $f_{\theta}(x)$ which is given by
\begin{align*}
\LT{ f_{\theta} }(t)
= \frac{\LTsub{X}(t+\theta)}{\LTsub{X}(\theta)}.
\end{align*}

The method described above approximates the exponentially tilted distribution.
This idea has been used in Asmussen \etal \cite{asmussen2016orthonormal} and Kang \etal \cite{KaPrRe2019}.
It is easily seen that taking $m > 1/(2\theta)$ implies that $(\e^{-\theta x} f_{X}(x))/f_{\nu}(x)\in\Lp^2(\nu)$. The ability to model claim sizes with heavy-tailed distributions is an improvement compared to Goffard \etal \cite{GoLoPo15,GoLoPo16} where only light-tailed distributions could be handled.

The expression of the \pdf in \eqref{eq:PseudoMixtureOfGammas} and \eqref{eq:polynomial_expansion_exponentially_tilted_pdf} resemble the one of an Erlang mixture, which are extensively used for risk modeling purposes, \cf Willmot and Woo \cite{WiWo07}, Lee and Lin \cite{LeLi10}, and Willmot and Lin \cite{WiLi11}. However, the $p_i$'s defined in \eqref{eq:PiExpression} do not form a proper probability mass function as they are not always positive. Hence our approximation cannot be considered as an approximation through an Erlang mixture although it enjoys the same features when it comes to approximating the survival function and the \slp as shown in the following result.
\begin{prop} \label{prop:OrthogonalPolynomialForm}
Letting $\Gamma_u(r,m,x)$ be the \svf of the $\mathsf{Gamma}(r,m)$ distribution,  we have:
\begin{itemize}
\item[(i)] the \svf of $X$ is given by
\begin{equation}\label{eq:TailFunctionX}
\overline{F}_{X}(x) = \sum_{i=0}^{\infty}p_i\Gamma_{u}(r+i,m,x) \quad \text{for } x \ge 0 \,,
\end{equation}
\item[(ii)] the usual \slp of $X$ with priority $a \ge 0$ is given by
\begin{equation}
\E\left[\left(X-a\right)_{+}\right] = \sum_{i=0}^{\infty}p_i\left[m (r+i) \Gamma_{u}(r+i+1,m,a)-a\Gamma_{u}(r+i,m,a)\right]. \label{eq:UsualStoplossX}
\end{equation}
\end{itemize}
\end{prop}
\begin{proof}
If $f_{X}/f_\nu\in \Lp^2(\nu)$ then Lemma~\ref{Lemma:PseudoMixtureOfGammas} allows us to write $f_X$ as in \eqref{eq:PseudoMixtureOfGammas}, and integrating this over $\left[x,\infty\right)$ yields the formula \eqref{eq:TailFunctionX}.
Now consider the usual \slp of $X$, and note that
\begin{align}
\E\left[(X-a)_+\right]=&\int_{a}^{\infty}(x-a)f_{X}(x)\dd x \nonumber\\
=&\int_{a}^{\infty}xf_{X}(x)\dd x-a\overline{F}_{X}(a).\label{eq:UsualStopLossXProof1}
\end{align}
Then notice that for every $i \in \N$, we have that
\begin{align}
\int_{a}^{\infty} x \, \gamma(r+i, m, x)\dd x
&= \int_{a}^{\infty}x\frac{x^{r+i-1} \e^{-x/m}}{\Gamma(r+i)m^{r+i}}\dd x\nonumber\\
&= m\frac{\Gamma(r+i+1)}{\Gamma(r+i)}\int_{a}^{\infty}\frac{x^{r+i} \e^{-x/m}}{\Gamma(r+i+1)m^{r+i+1}}\dd x\nonumber\\
&= m (r+i) \Gamma_{u}(r+i+1,m,a).\label{eq:UsualStopLossXProof2}
\end{align}
Therefore substituting \eqref{eq:PseudoMixtureOfGammas} and \eqref{eq:TailFunctionX} into \eqref{eq:UsualStopLossXProof1} and simplifying with \eqref{eq:UsualStopLossXProof2} yields \eqref{eq:UsualStoplossX}.
\end{proof}

Proposition \ref{prop:OrthogonalPolynomialForm} represents a practical refinement in comparison to the works Goffard \etal \cite{GoLoPo15,GoLoPo16} as the formulas derived may be readily evaluated without using numerical integration.

\subsection{Approximating densities of positive compound distributions}
We now focus on variables $S_N$ which admit a compound distribution. Since these distributions have an atom at 0, we put aside this singularity and focus on the defective \pdf $f_{S_N}^+$. The discussion in Sections \ref{ssec:PolynomialApproxIntro} and \ref{ssec:PolynomialApproxPositive} also apply to defective densities. Namely, if $f_{S_N}^{+}/f_\nu\in \Lp^2(\nu)$ then the expansion in Lemma \ref{Lemma:PseudoMixtureOfGammas} is valid, and we have
\begin{equation*}\label{eq:PseudoErlangMixtureRepresentationDefectivePDF}
f_{S_N}^+(x)= \sum_{k=0}^\infty q_k Q_k(x) \, \gamma(r, m, x) = \sum_{i=0}^{\infty} p_i \gamma(r+i,m,x),\text{ for }x>0,
\end{equation*}
where $q_k=\int_{0}^{\infty}Q_{k}(x)f_{S_N}^+(x)\dd x$ and $p_i$ is given by \eqref{eq:PiExpression}.
Truncating the first summation yields
\begin{equation*}\label{eq:ApproxPolynomialRepresentation}
f_{S_N}^{+}(x)
\approx \sum_{k=0}^{K} q_k Q_{k}(x) \, \gamma(r,m,x)
= \sum_{i=0}^{K} \widehat{p}_i\gamma(r+i,m,x),
\end{equation*}
where $\widehat{p}_i =\sum_{k=i}^K q_k (-1)^{i+k} / [i! \, (k-i)! ] \sqrt{k! \Gamma(k+r) / \Gamma(r)}$ for $i \le K$. Evaluations of the survival function $\overline{F}_{S_N}$ and the \slp $\E\left[\left(S_N-a\right)_{+}\right]$ follow from Proposition \ref{prop:OrthogonalPolynomialForm}. If the integrability condition is not satisfied then the exponentially tilted version of the defective \pdf is expanded.

\subsubsection{Choice of $r$ and $m$}\label{ssec:ChoosingmAndr}

The parameters for the polynomial approximations are set differently for the light-tailed and heavy-tailed cases. In the light-tailed cases moment matching of order $2$ is the natural procedure to set the values of $r$ and $m$. We need to take into account the result in Proposition \ref{prop:ExpBound} and make sure that $m>1/(2\rho_X)$. Hence, the value of $\rho_X$ depends on the distributions of $N$ and $U$. The two distributions we use for modeling the claim frequency $N$ are the \textit{Poisson} and the \textit{Pascal} distributions. The Poisson distribution is denoted by $\mathsf{Poisson}(\lambda)$ with \pmf
\[ f_N(k) = \frac{e^{-\lambda}\lambda^{k}}{k!} \,, \quad \text{ for }k = 0,1\dots \,, \]
where $\lambda>0 $. We define the Pascal \rv to be the number of failures counted before observing $\alpha \in \Npos$ successes, denoted $\mathsf{Pascal}(\alpha,p)$ with \pmf
\[ f_N(k) = \binom{\alpha+k-1}{k} p^\alpha q^k \,, \quad \text{ for } k = 0,1,\dots \,. \]
Our method is applicable for any claim frequency distribution as long it admits a probability generating function. This allows us to compute the expansion coefficients, detailed later in Section~\ref{sssec:ComputingOrthogonalCoefficients}.
\begin{ex}\label{ex:ParametersCompoundPoisson}
Let $N$ be Poisson distributed, the moment generating function of $S_N$ is then given by
$$
\mathcal{L}_{S_N}(-s)=\exp\left[\lambda(\mathcal{L}_U(-s)-1)\right].
$$
The radius of convergence of $S_N$ coincides with the one of $U$, $\rho_{S_N}=\rho_U$. In  that case, we can set $r=1$ and m=$\lambda\mathbb{E}(U)$ which corresponds to a moment matching procedure of order 1 or set $r=\lambda\mathbb{E}(U)^{2}/\mathbb{E}\left(U^{2}\right)$ and $m=\mathbb{E}\left(U^{2}\right)/\mathbb{E}(U)$ which, in turns, matches the two first moments.
\end{ex}
\begin{ex}\label{ex:ParametersCompoundPascal}
Let $N$ be Pascal distributed, the moment generating function of $S_N$ is then given by
$$
\mathcal{L}_{S_N}(-s)=\left[\frac{p}{1-q\mathcal{L}_U(-s)}\right]^{\alpha}.
$$
The radius of convergence $\rho_{S_N}$ is the positive solution of the equation $\mathcal{L}_U(-s)=q^{-1}$. We set $r=1$ and $m=\rho_{S_N}^{-1}$.
\end{ex}
The parametrization proposed in Example \ref{ex:ParametersCompoundPascal} is linked to the fact that it leads to the exact defective \pdf in the case of a compound Pascal model with exponentially distributed claim sizes. First, we need to introduce the binomial distribution denoted by $\mathsf{Binomial}(n,p)$ with \pmf
 \[ f_N(k) = \binom{n}{k} p^k q^{n-k} \,, \quad \text{ for }k = 0,1,\dots, n\,, \]
 where $p \in (0,1)$, $n \in \Npos$, and $p+q=1$. The following lemma, adapted from \cite{PaWi81}, shows a useful correspondence between the Pascal and binomial distributions when used in compound sums with the exponential distribution.
\begin{lemma} \label{lemma:BinomialPascalEquivalence}
Consider the random sums $X = \sum_{i=1}^{N_1} U_i$ and $Y = \sum_{i=1}^{N_2} V_i$, where
\[
N_1 \sim \mathsf{Pascal}(\alpha,p) \,, \quad  U_i \iidDist \mathsf{Gamma}(1,\beta) \,, \quad  N_2 \sim \mathsf{Binomial}(\alpha,q)  \,, \quad V_i \iidDist \mathsf{Gamma}(1,p^{-1} \beta) \,,
\]
where $p \in (0,1)$, $\alpha \in \Npos$, $p + q = 1$, and where $\beta > 0$.
Then we have $X \eqdistr Y$.
\end{lemma}
\begin{proof}
Both $X$ and $Y$ have the same Laplace transform, so $X \eqdistr Y$.
\end{proof}

\begin{coro} \label{Coro:PascalExponentialStopLoss}
Consider the compound sum $S_N = \sum_{i=1}^N U_i$ where $N \sim \mathsf{Pascal}(\alpha,p)$ and the $U_i \iidDist\mathsf{Gamma}(1,\beta)$.
Then the \svf of $S_N$ is given by
\begin{align*}
\overline{F}_{S_N}(x) &= \sum_{i=1}^\alpha \binom{\alpha}{i} q^i p^{\alpha-i} \, \Gamma_u\left(i, p^{-1}\beta, x\right),
\end{align*}
and its \slp is given by
\begin{equation*}\label{eq:StopLossCompoundNegBinExp}
\E\left[\left(S_N-a\right)_{+}\right] = \sum_{i=1}^\alpha \binom{\alpha}{i} q^i p^{\alpha-i} \left[ \frac{i\beta}{p } \Gamma_u\big(i+1, p^{-1} \beta , a \big) - a \Gamma_u\left(i, p^{-1}\beta , a\right) \right].
\end{equation*}
\end{coro}
\begin{proof}
By Lemma~\ref{lemma:BinomialPascalEquivalence} we can instead consider the $S_N$ defined by $N \sim \mathsf{Binomial}(\alpha,q)$ and with $U_i \iidDist \mathsf{Gamma}(1,p^{-1} \beta)$.
Noting that $S_n = U_1 + \dots + U_n \sim \mathsf{Gamma}(n, p^{-1} \beta)$ gives the result.
\end{proof}
One conclusion of Corollary~\ref{Coro:PascalExponentialStopLoss} is that the exact solution coincides with our approximation when $r=1$ and $m=p^{-1}\beta$ (and with $K \ge \alpha - 1$). Note that $p\beta^{-1}$ is the solution of the equation $\LTsub{U}(-s)=q^{-1}$ which is consistent with the parametrization proposed in Example \ref{ex:ParametersCompoundPascal}.\\

In the heavy-tailed cases we set $\theta=1$, $m~=~\theta/2~=~1/2$ (at the lower limit for $m$; this gives $\widetilde{m}=1$), and $r = \E[U]$.

\subsubsection{Computation of the $q_k$'s} \label{sssec:ComputingOrthogonalCoefficients}

The inherent challenge of the implementation of the polynomial method remains the evaluation of the coefficients $\{q_k, k\geq0\}$. Recall that
\begin{equation*}
q_{k}=\int_{0}^{\infty}Q_k(x)f_{S_N}^+(x)\dd x,\text{ }k\geq0.
\end{equation*}
We propose an evaluation based on the Laplace transform $\LT{f_{S_N}^+}$. Define the generating function of the sequence $\{q_k c_k, k\geq0\}$ as $\mathcal{Q}(z)=\sum_{k=0}^{\infty}q_{k}c_{k}z^{k}$, where
\begin{equation*}\label{eq:ck}
c_k= \left(\frac{\Gamma(k+r)}{\Gamma(k+1)\Gamma(r)}\right)^{1/2}, \quad \text{ for }k\geq0 \,.
\end{equation*}
The following result establishes a link between the Laplace transform of $f_{S_N}^+$  and the generating function $\mathcal{Q}(z)$.
\begin{prop}\label{prop:LaplaceTransformPolynomialRepresentation}
Assume that $f_{S_N}^+/f_{\nu}\in\Lp^{2}(\nu)$, then we have
\begin{equation}\label{eq:CoefficientGeneratingFunctionToLaplaceTransforn}
\mathcal{Q}(z)=(1+z)^{-r}\LT{f_{S_N}^+}\Big[\frac{-z}{m(1+z)}\Big].
\end{equation}
\end{prop}
\begin{proof}
As $f_{S_N}^+/f_{\nu}\in\Lp^{2}(\nu)$, the polynomial representation of $f_{S_N}^+$ follows from the application of Lemma~\ref{Lemma:PseudoMixtureOfGammas} with
\begin{equation}\label{eq:prop:PolynomialRepresentationgX}
f_{S_N}^+(x)=\sum_{k=0}^{\infty}\sum_{i=0}^{k}q_k\frac{(-1)^{i+k}}{i! \, (k-i)!} \sqrt{\frac{k! \Gamma(k+r)}{\Gamma(r)}}\gamma(r+i,m,x).
\end{equation}
Taking the Laplace transform in \eqref{eq:prop:PolynomialRepresentationgX} yields
\begin{align*}
\LT{f_{S_N}^+}(s)
&= \left(\frac{1}{1+sm}\right)^{r}\sum_{k=0}^{\infty}q_k \sum_{i=0}^{k} (-1)^{k+i} \left(\frac{\Gamma(k+r)}{\Gamma(k+1)\Gamma(r)}\right)^{1/2} \binom{k}{i} \left(\frac{1}{1+sm}\right)^{i}\nonumber \\
&= \left(\frac{1}{1+sm}\right)^{r}\sum_{k=0}^{\infty}q_k c_k (-1)^{k} \sum_{i=0}^{k} \binom{k}{i} \left(\frac{-1}{1+sm}\right)^{i} \nonumber\\
&= \left(\frac{1}{1+sm}\right)^{r}\sum_{k=0}^{\infty}q_k c_k (-1)^{k} \left( \frac{sm}{1+sm}  \right)^{k}\nonumber \\
&= \left( 1 - \frac{sm}{1+sm}\right)^{r}\mathcal{Q}\left(-\frac{sm}{1+sm}\right).
\end{align*}
Thus \eqref{eq:CoefficientGeneratingFunctionToLaplaceTransforn} follows from letting $z= -sm / (1+sm)$.
\end{proof}
The Laplace transform of the defective \pdf $f_{S_N}^{+}$ is given by
$$
\LT{f_{S_N}^{+}}(s)=\mathcal{L}_{S_N}(s)-\mathbb{P}(N=0).
$$
The coefficients of the polynomials can be derived after differentiation of the generating function $\mathcal{Q}(z)$ as
\begin{align*}\label{eq:PolynomialExpansionCoefficientDerivativeGeneratingFunction}
q_k&=\frac{1}{c_k}\frac{1}{k!}\frac{\dd^{k}}{\dd z^{k}}\mathcal{Q}(z)\Big\rvert_{z=0}  =\frac{1}{c_k} \text{Coefficient}(k, \text{MaclaurinSeries}(\mathcal{Q}(z))).
\end{align*}

\section{Laplace transform inversion approximations}\label{sec:NumericalInversionLaplaceTransform}
We present in this section a method inspired from the work of Abate and Whitt \cite{Abate1992} to recover the survival function of a compound distribution from the knowledge of its Laplace transform. The methodology is further applied to the computation of \slps by taking advantage of the connection between the \slp of $S_N$ and the survival function of the equilibrium distribution of $S_N$. This well established method is recalled here for the sake of self-containedness and adapted to our notation. We begin by stating some useful transform relations, then discuss the general Laplace inversion framework that we will use, and will apply the method to the compound distribution problem.

\subsection{Numerical Laplace inversion}

A function $f$ can be recovered from its Laplace transform by a standard Bromwich integral.
We assume $f:\,\R_+ \to \R_+$, is a measurable function with locally bounded variation.
To define the Bromwich integral, first select a $\gamma > 0$ (we discuss this choice later), then
\[
f(x)=\frac{2\e^{\gamma x}}{\pi}\int_{0}^{\infty}\cos(xs)\Re\left[\LT{f}(\gamma + is)\right]\dd s.
\]

We apply a basic numerical integration system to this integral by first \textit{discretizing} the integral and then \textit{truncating} the resulting infinite sum. In both steps, we follow the steps of Abate and Whitt \cite{Abate1992}.

\subsubsection{Discretization} \label{Sub:Discretization}

We will use a semi-infinite trapezoidal rule, despite the apparent simplicity of the method.
With a grid size $h>0$, this discretization yields
\[
f(x) \approx f_{\text{disc}}(x) \defeq \frac{2\e^{\gamma x}}{\pi} \cdot h \, \Big \{ \frac12 \LT{f}(\gamma) + \sum_{j=1}^\infty  \cos(x \cdot hj)\Re\left[\LT{f}(\gamma + i \cdot hj)\right] \Big \},
\]
since  $\Re\left[\LT{f}(\gamma)\right] = \LT{f}(\gamma)$. We simplify this by choosing $h = \pi/(2 x)$ and $\gamma = a / (2 x)$ for an $a > 0$, achieving
\begin{equation} \label{DiscretizedInversion}
f_{\text{disc}}(x) = \frac{\e^{a/2}}{2x}\LT{f} \left(\frac{a}{2 x}\right) + \frac{\e^{a/2}}{x} \sum_{k=1}^\infty (-1)^{k} \Re\left[\LT{f}\left(\frac{a+ i \cdot 2\pi k}{2 x}\right)\right] \,.
\end{equation}

From Theorem 5.5.1 of \cite{RoScScTe08} we have that the \textit{discretization error} (also called \textit{sampling error}) is simply
\begin{equation} \label{DiscretizationError}
    f_{\text{disc}}(x) - f(x) = \sum_{k=1}^\infty \e^{-a k} f\big( (2k+1) x \big) \,.
\end{equation}
In particular, if $0 \le f(x) \le 1$, then
\begin{equation} \label{BoundedDiscretizationError}
    f_{\text{disc}}(x) - f(x) \le \frac{\e^{-a}}{1-\e^{-a}} \,.
\end{equation}

There are no absolute value signs here --- the discretization introduces a systematic overestimate of the true function value.
Also, \eqref{DiscretizationError} implies that $a$ should be as large as possible (limited eventually by finite-precision computation).
The benefit of knowing this result is slightly offset by the requirement that $h$ and $\gamma$ now be functions of $x$ rather than constants.

\subsubsection{Truncation}

Due to the infinite series, the expression in \eqref{DiscretizedInversion} cannot be directly computed, thus it has to be truncated.
The arbitrary-seeming choice of $h$ and $\gamma$ in Section~\ref{Sub:Discretization} not only allows for calculation of the discretization error, but also benefits the truncation step. This is because the sum in \eqref{DiscretizedInversion} is (nearly) of alternating sign, and thus \textit{Euler series acceleration} can be applied to decrease the truncation error. Define for $\ell=1,2,\dots$
\begin{equation*} \label{TruncatedInversion}
s_\ell(x) \defeq \frac{\e^{a/2}}{2x}\LT{f} \left(\frac{a}{2 x}\right) + \frac{\e^{a/2}}{x} \sum_{k=1}^\ell (-1)^{k} \Re\left[\LT{f}\left(\frac{a+ i \cdot 2\pi k}{2 x}\right)\right] \,.
\end{equation*}
Then, for some positive integers $M_1$ and $M_2$,
\begin{equation}\label{eq:FinalApproxLaplaceInversion}
f(x) \approx f_{\text{disc}}(x) \approx f_{\text{approx}}(x) \defeq \sum_{k=0}^{M_1} \binom{M_1}{k} 2^{-M_1} s_{M_2+k}(x) \,.
\end{equation}

\subsection{Approximations of the survival function and stop-loss premium for compound distributions}

For a random sum  $S_N$, we consider using the technique above to evaluate the \svf $\overline{F}_{S_N}$ and the \slps from their Laplace transform.
We invert $\LT{\overline{F}_{S_N}}$, but note that inverting $\LT{F_{S_N}}$ produces almost identical results.

This inversion easily gives approximations of $\overline{F}_{S_N}$, though evaluating the \slps requires extra thought.
As noted in Dufresne \etal \cite{DuGaMo09}, we have that
\begin{equation}\label{eq:LinkStopLossPremiumEquilibriumDistribution}
\mathbb{E}\left[(S_N-d)_{+}\right]=\mathbb{E}(S_N)\overline{F_{S_N^{\ast}}}(d),
\end{equation}
where $S_N^{\ast}$ is a random variable under the \textit{equilibrium distribution} with density
\begin{equation*}
f_{S_N^{\ast}}(x)=\begin{cases}
\overline{F}_{S_N}(x) / \mathbb{E}(S_N),&\text{ for }x>0,\\
0,&\text{ otherwise},
\end{cases}
\end{equation*}
and Laplace transform
\begin{equation*}\label{eq:LaplaceTransformEquilibriumDistribution}
\LTsub{S_N^{\ast}}(s)=\frac{1-\LTsub{S_N}(s)}{s\mathbb{E}(S_N)}.
\end{equation*}
The \slp is then obtained, replacing in \eqref{eq:LinkStopLossPremiumEquilibriumDistribution} the \svf of $S_N^{\ast}$ by its approximation in \eqref{eq:FinalApproxLaplaceInversion}.

\section{Numerical illustrations}\label{sec:NumericalIllustrations}
We illustrate the performance of the two proposed numerical procedures. Section \ref{subsec:NumericsSurvivalFunctionStopLoss} focuses on approximating the \svf and the \slp associated to aggregated claim sizes, while Section \ref{subsec:ApproximationFiniteTimeRuinProbability} considers the approximation of the finite-time ruin probability with no initial reserves using formula \eqref{eq:ConnectionFiniteTimeRuinProbabilityStopLossPremium}.

For each test case, we compare the orthogonal polynomial approximation, the Laplace inversion approximation, and the crude Monte Carlo approximation. When $U$ is gamma distributed, we use the fact that $S_n$ is Erlang distributed to produce an approximate distribution for $S_N$ by truncating $N$ to be less than some large level.

The parameters for the polynomial approximations has been discussed in Section \ref{ssec:ChoosingmAndr}, the calibration is depending on the assumptions over the claim frequency and claim sizes distribution.
The parameters for the Laplace inversion technique are set to $M_1=11$, $M_2=15$ and $a=18.5$ following the example of Rolski \etal \cite[Chapter 5, Section 5]{RoScScTe08}. This choice of $a$ implies that the discretization error is less than $10^{-8}$, derived from \eqref{BoundedDiscretizationError}. We do not use any built-in routines for the Laplace inversion, but simply implement \eqref{eq:FinalApproxLaplaceInversion}.

In each plot, the first subplot shows the approximations each method produces, and the second shows the \textit{approximate absolute error}. We define this, for method $i \in \{1,\dots,I\}$, as
\begin{align*}
\text{ApproximateAbsoluteError}( \widehat{f}_i, x )
&:= \widehat{f}_i(x)- \text{Median}\big\{ \widehat{f}_1(x),\dots,\widehat{f}_I(x) \big\} \\
&\approx \widehat{f}_i(x)-f(x) =: \text{AbsoluteError}( \widehat{f}_i, x ) \,.
\end{align*}

To create very accurate orthogonal polynomial approximations we let the truncation parameter $K$ be 16 which is quite large. The coefficients for this expansion are determined by symbolically calculating a Taylor series expansion of order $K$. As \textsc{Mathematica} has one of the most advanced symbolic calculus engines available we use this language. We replicated some of the tests in \textsc{Python} using the open-source \texttt{Sympy} symbolic mathematics library (specifically the faster \texttt{Symengine} version which is implemented in \textsc{C++}), though the derivatives of some special functions which appear in our test cases (e.g.\ in the Laplace transform of $S_N$ in Test 3) are not implemented yet. Both implementations are available online \cite{StoplossCode}.

\subsection{Survival function and stop-loss premium computations}

\label{subsec:NumericsSurvivalFunctionStopLoss}

To ensure both methods were implemented correctly, we applied them to the case where $N\sim\mathsf{Pascal}(\alpha=10,p=3/4)$ and $U\sim\mathsf{Gamma}(r=1,m=1/6)$. Corollary~\ref{Coro:PascalExponentialStopLoss} tells us the orthogonal approximation (with $r=1$, $m=(1/6)/(3/4)=2/9$ and $K = 10-1 = 9$) is equivalent to the true function, which we verified, and the Laplace inversion errors in Tables \ref{tbl:PascalExponentialSVFTest} and \ref{tbl:PascalExponentialSLPTest} are acceptably small.

\begin{table}[H]
\centering
\caption{Relative errors for the Laplace inversion \svf approximation} \vspace{1em}
\begin{tabular}{c|ccccc}
$x$                   & 0.5  & 1  & 1.5  & 2 & 2.5  \\
\hline
Error & 7.27e-7 & 1.92e-6 & 5.86e-6 & 1.78e-5 & 4.01e-5 \\
\end{tabular}
\label{tbl:PascalExponentialSVFTest}

\caption{Relative errors for the Laplace inversion \slp approximation} \vspace{1em}
\begin{tabular}{c|ccccc}
$a$                   & 0.5  & 1  & 1.5  & 2 & 2.5  \\
\hline
Error & 8.68e-7 & 2.27e-6 & 5.92e-6 & 1.12e-5 & -2.12e-5 \\
\end{tabular}

\label{tbl:PascalExponentialSLPTest}
\end{table}

\begin{test}
$N\sim\mathsf{Poisson}(\lambda=2)$, and $U\sim\mathsf{Gamma}(r=3/2,m=1/3)$
\end{test}

\begin{figure}[H]
\centering
\includegraphics[width=0.95\textwidth]{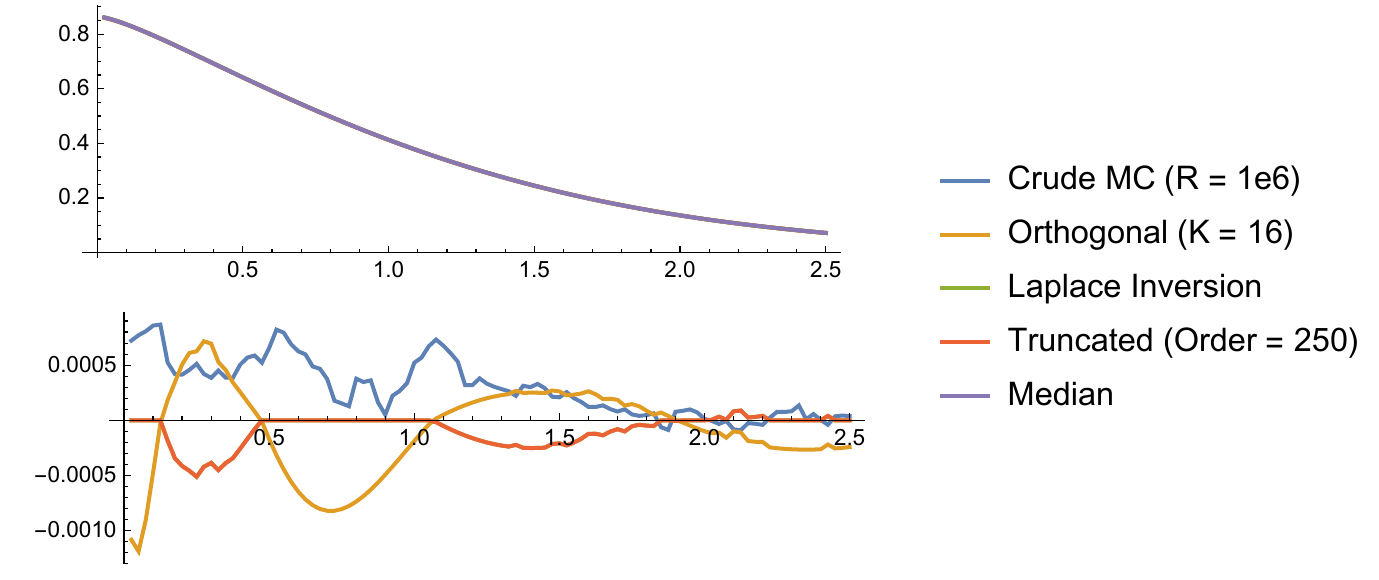}
\caption{Survival function approximation and approximate absolute error for Test 1.}
\end{figure}

\begin{figure}[H]
\centering
\includegraphics[width=0.95\textwidth]{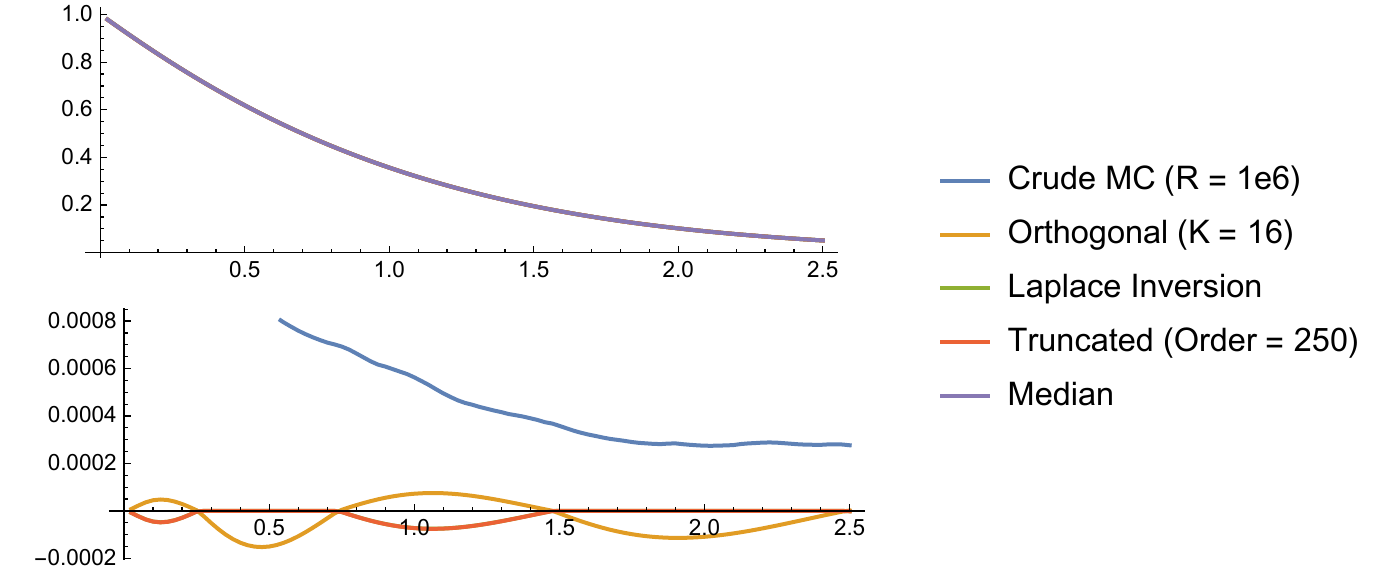}
\caption{Stop-loss premium approximation and approximate absolute error for Test 1.}
\end{figure}

\begin{test}
$N\sim\mathsf{Pascal}(\alpha=10,p=1/6)$, and $U\sim\mathsf{Gamma}(r=3/2,m=1/75)$
\end{test}

This test case (up to the scaling constant) has been considered by Jin \etal \cite[Example 3]{JiPrRe16}. In the plots for this test case, the orthogonal expansion, the Laplace inversion method, and the truncation all give the same values and hence are hidden underneath the last of these approximations to be plotted.

\begin{figure}[H]
\centering
\includegraphics[width=0.95\textwidth]{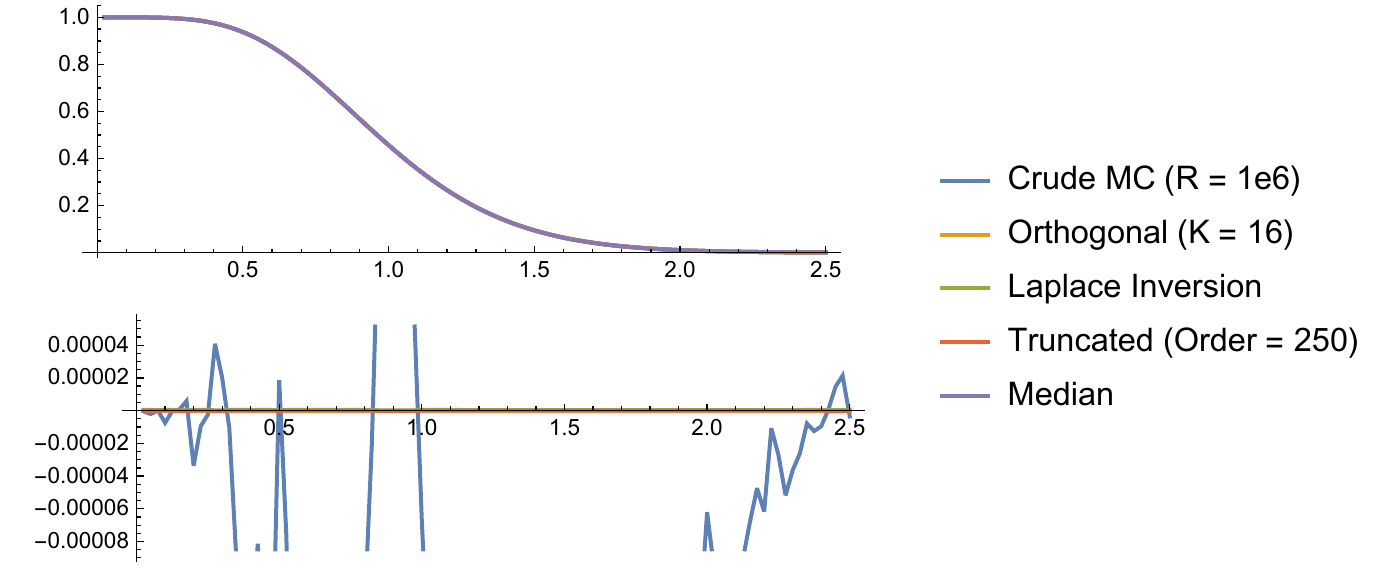}
\caption{Survival function approximation and approximate absolute error for Test 2.}
\end{figure}

\begin{figure}[H]
\centering
\includegraphics[width=0.95\textwidth]{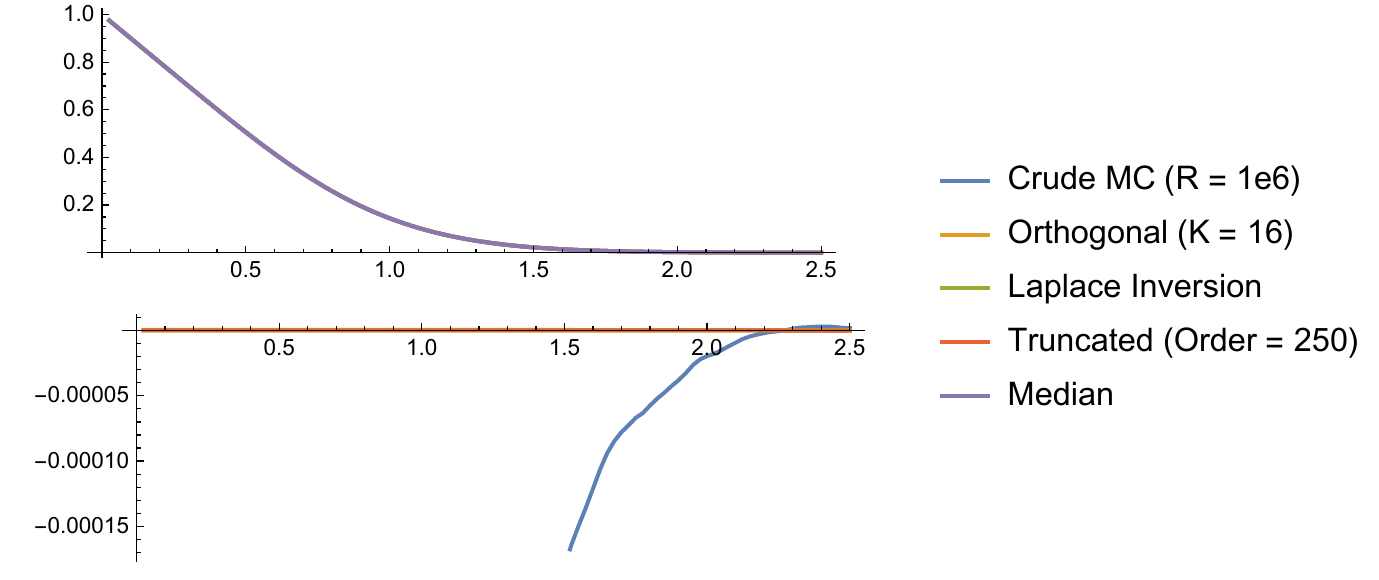}
\caption{Stop-loss premium approximation and approximate absolute error for Test 2.}
\end{figure}

\begin{test} $N\sim\mathsf{Poisson}(\lambda=4)$, and $U\sim\mathsf{Pareto}(a=5,b=11,\theta=0)$
\end{test}
The survival function for $U$, given $x \ge \theta=0$, is
\[ \overline{F}_U(x) = \left( 1 + \frac{x-\theta}{a} \right)^{-b} = \left( 1 + \frac{x}{5} \right)^{-11} \,. \]
We note that the Laplace inversion approximator breaks down for small values of $x$ or $a$ in this test case. The specific error given is an ``out of memory'' exception when \textsc{Mathematica} is attempting to do some algebra with extremely large numbers. It is unclear whether a different implementation or selection of parameters would fix this behaviour.

\begin{figure}[H]
\centering
\includegraphics[width=0.95\textwidth]{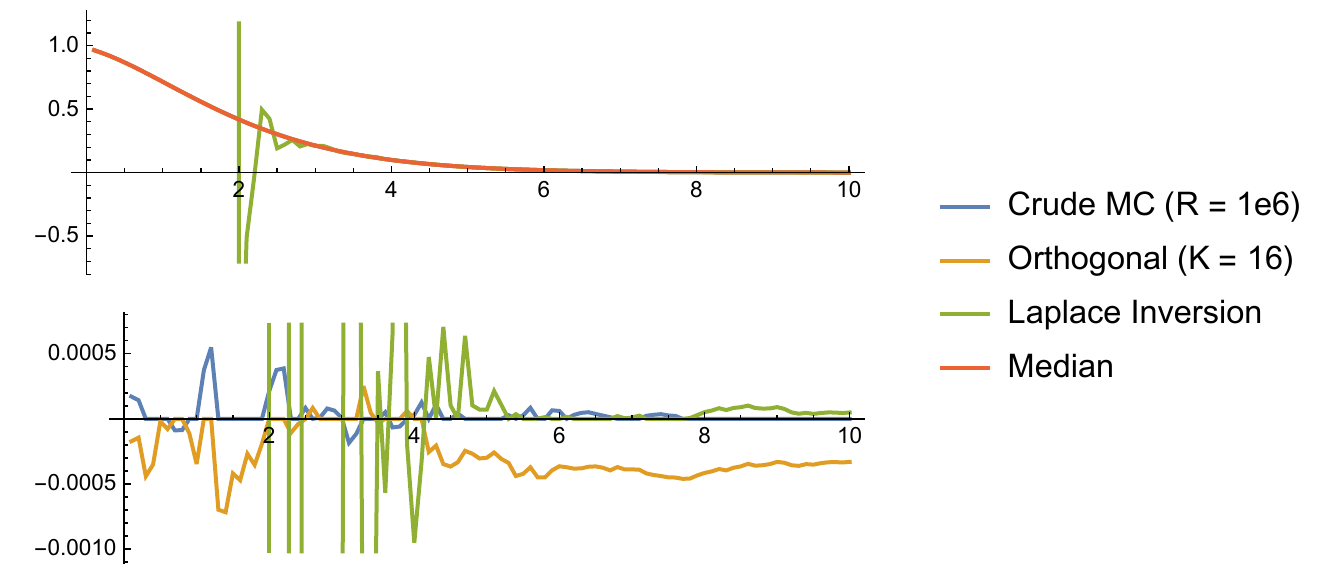}
\caption{Survival function approximation and approximate absolute error for Test 3.}
\end{figure}

\begin{figure}[H]
\centering
\includegraphics[width=0.95\textwidth]{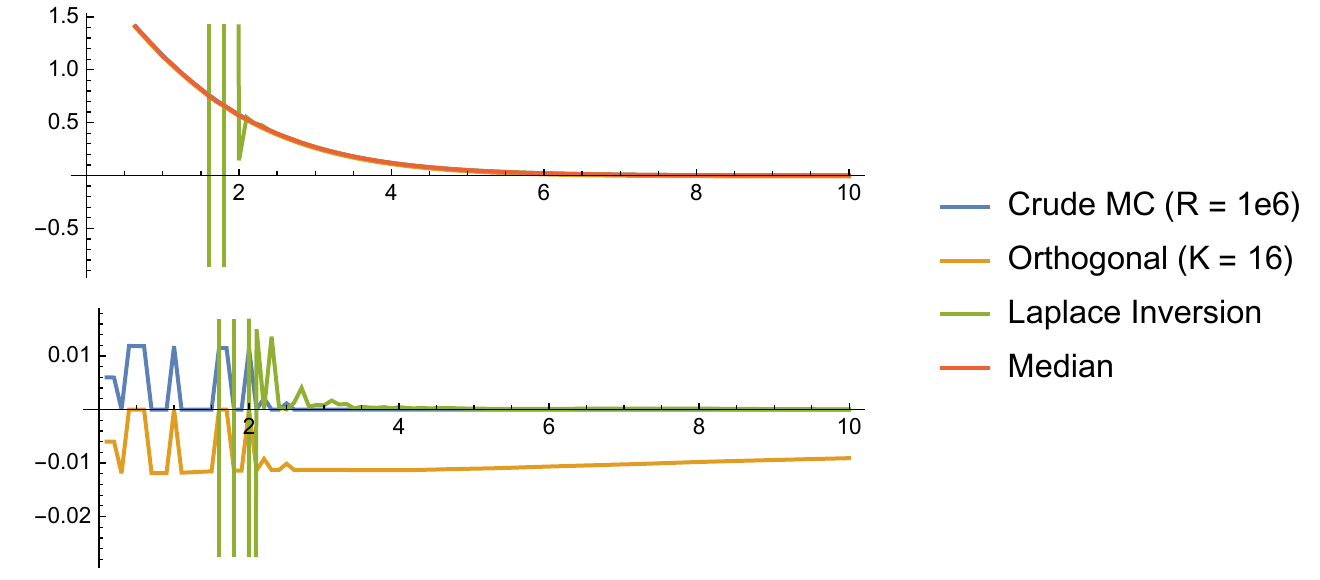}
\caption{Stop-loss premium approximation and approximate absolute error for Test 3.}
\end{figure}

\begin{test} $N\sim\mathsf{Pascal}(\alpha=2,p=1/4)$, and $U\sim\mathsf{Weibull}(\beta=1/2,\lambda=1/2)$
\end{test}
The survival function for $U$, given $x \ge 0$, is
\[ \overline{F}_U(x) = \exp\left\{ {-} \left(\frac{x}{\lambda} \right)^\beta \right\} = \exp\left\{ {-} \sqrt{2 x}  \right\} \,. \]

\begin{figure}[H]
\centering
\includegraphics[width=0.95\textwidth]{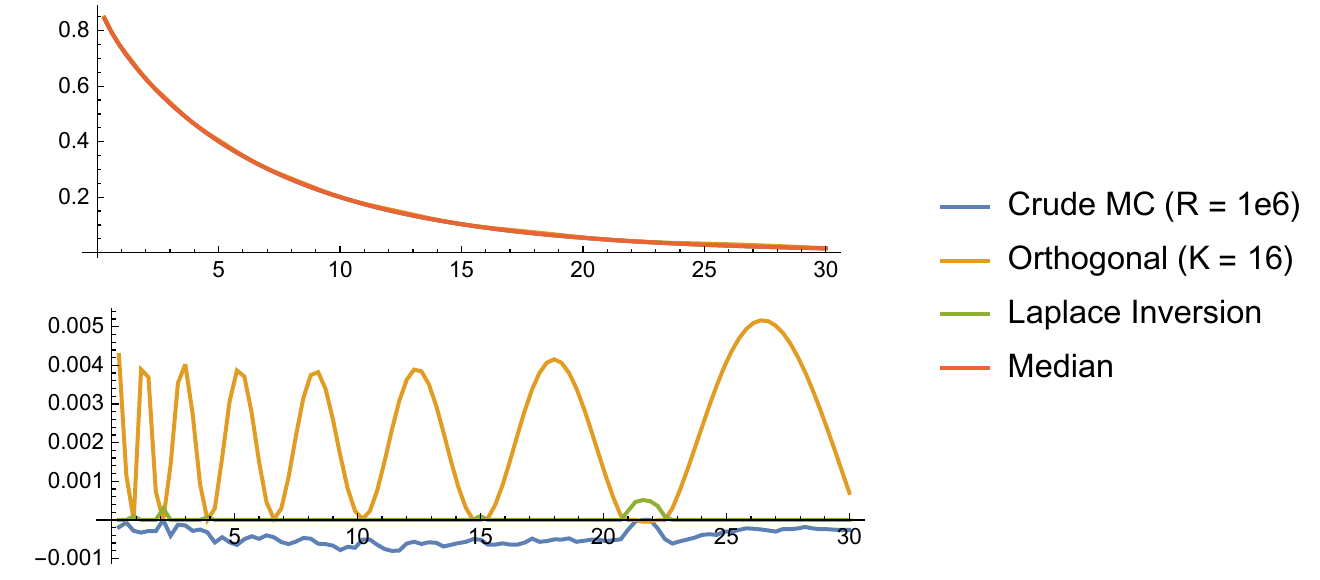}
\caption{Survival function approximations and approximate absolute error for Test 4.}
\end{figure}

\begin{figure}[H]
\centering
\includegraphics[width=0.95\textwidth]{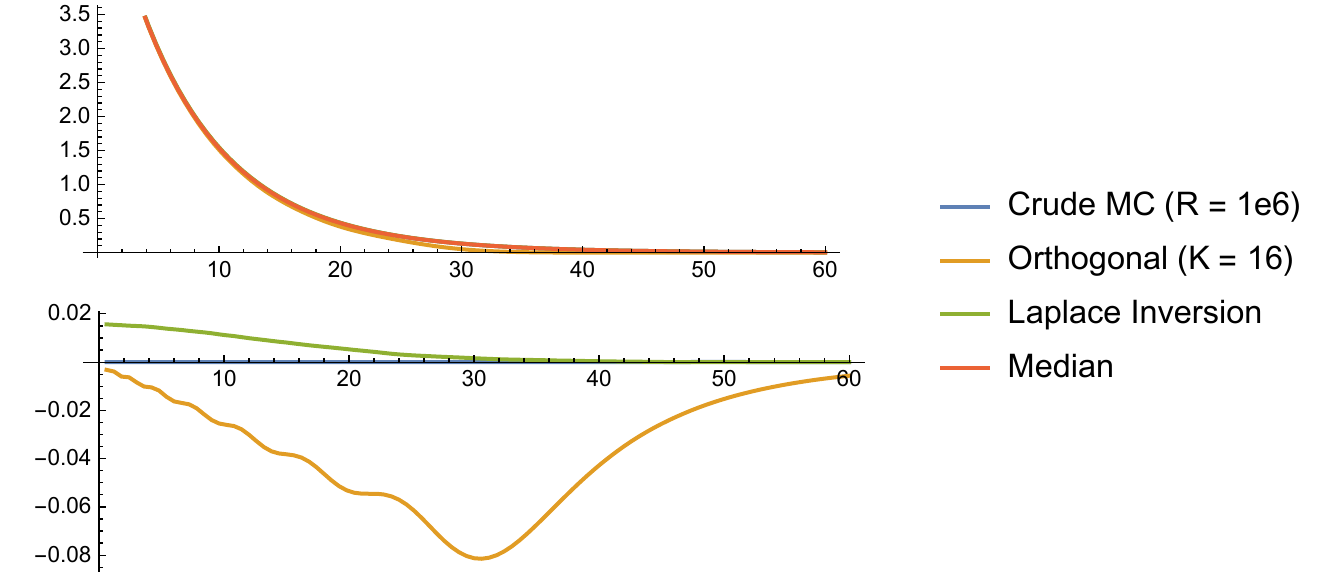}
\caption{Stop-loss premium approximations and approximate absolute error for Test 4.}
\end{figure}

\subsection{Finite-time ruin probability with no initial reserve} \label{subsec:ApproximationFiniteTimeRuinProbability}

The plots above have used common random numbers for smoothing purposes, however this is not possible in the following plots so they will appear less smooth.

\begin{test}
$\lambda=4$ and $U\sim\mathsf{Gamma}(r=2,m=2)$ and $c=1$
\end{test}

\begin{figure}[H]
\centering
\includegraphics[width=0.95\textwidth]{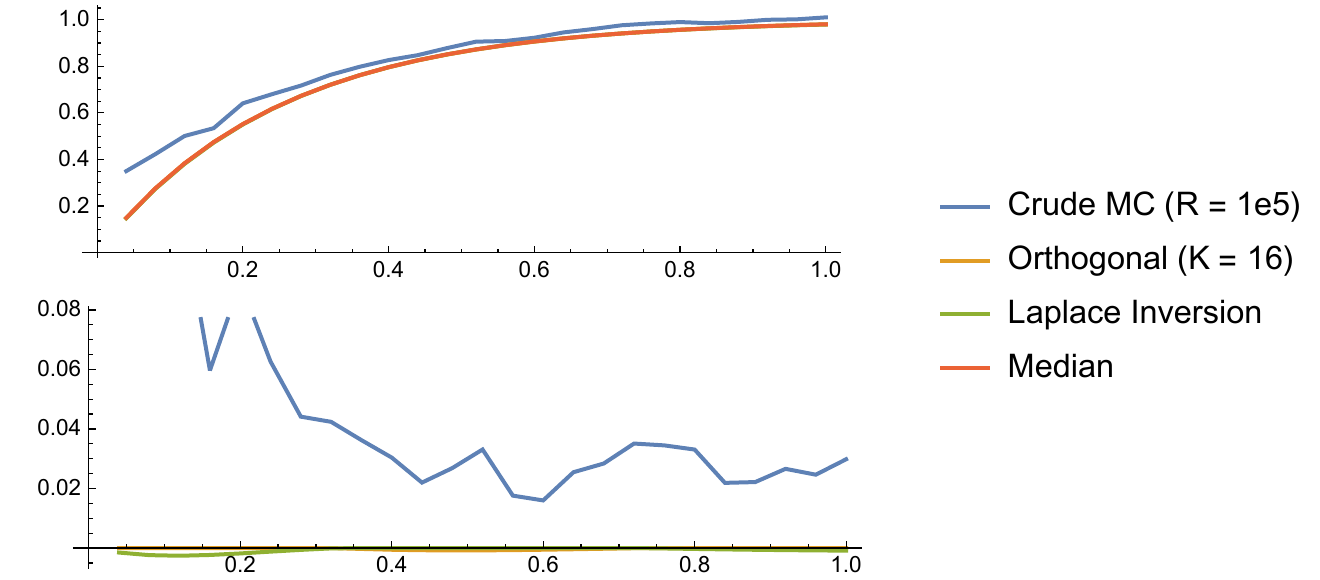}
\caption{Ruin probability $\psi(0, t)$ approximations and approximate absolute error for Test 5.}
\end{figure}

\begin{test} $\lambda=2$ and $U\sim\mathsf{Pareto}(a=5,b=11,\theta=0)$ and $c=1$
\end{test}

See the discussion of Test 3 for a description of the Laplace inversion formula's poor behaviour when Pareto variables are involved.

\begin{figure}[H]
\centering
\includegraphics[width=0.95\textwidth]{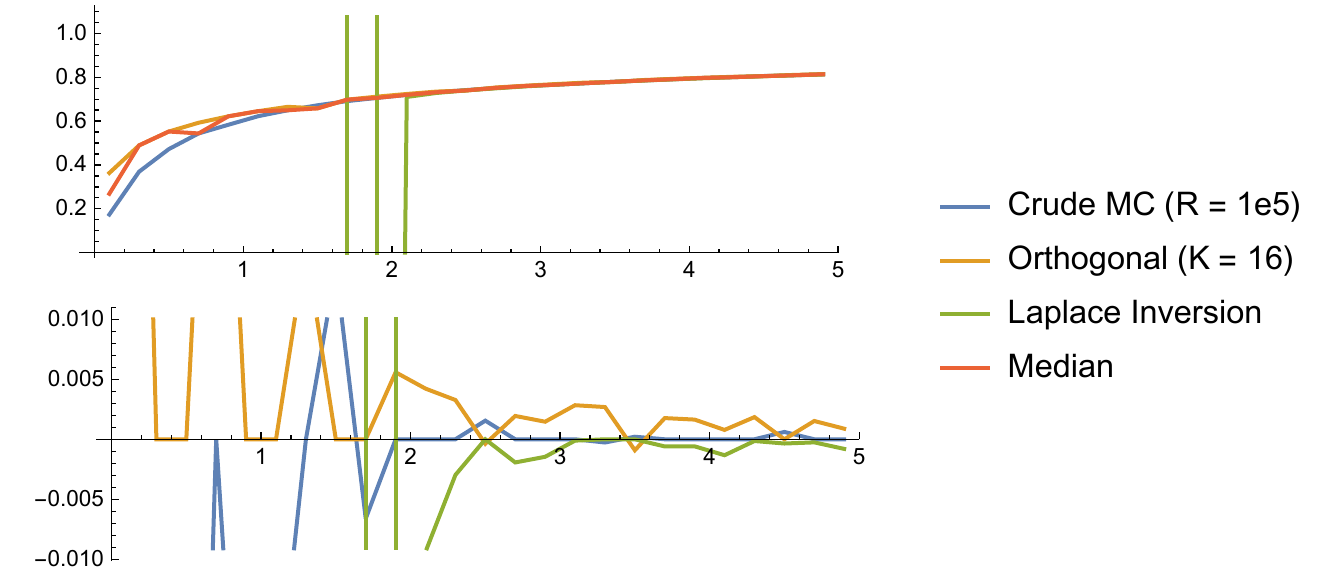}
\caption{Ruin probability $\psi(0, t)$ approximations and approximate absolute error for Test 6.}
\end{figure}
\subsection{Concluding remarks}
The orthogonal polynomial method has performed well across all the test cases studied. The accuracy is acceptable even with a rather small order of truncation $K=16$. It produces an approximation having an analytical expression, which is desirable, and in a timely manner. The precision may be improved by adding more terms in the expansions. The main drawback is probably the need for a parametrization tailored to the case studied.

The Laplace transform inversion method yields outstanding result in terms of accuracy. It failed to provide a stable approximation for Pareto distributed claim sizes. The parametrization is automatic and seems to fit the different case studied (except the Pareto one).

The main conclusion is that both methods are easy to implement and are superior to a simple truncation or a crude Monte Carlo approach.

The approximation formulas proposed in our paper may be turned into a nonparametric estimator of the density. One could substitute the coefficients within the polynomial expansion by their empirical counterparts if data were available. This extension will be at the center of a forthcoming research project.

\section*{Acknowledgments}

The authors are grateful to the reviewers for their careful reading and insightful comments.
This research was conducted both at l'Université Claude Bernard Lyon 1 and at the author's previous institutions, whose support must be acknowledged.
Pierre-Olivier Goffard was partially funded by a Center of Actuarial Excellence educational grant given to the University of California, Santa Barbara, by the Society of Actuaries. Patrick J.\ Laub was supported by an Australian Government Research Training Program Scholarship and an Australian Research Council Centre of Excellence for Mathematical \& Statistical Frontiers Scholarship.


\begin{thebibliography}{10}

\bibitem{Abate1992}
Joseph Abate and Ward Whitt.
\newblock The {F}ourier-series method for inverting transforms of probability
  distributions.
\newblock {\em Queueing Systems}, 10(1):5--87, 1992.

\bibitem{asmussen2010ruin}
S{\o}ren Asmussen and Hansj{\"o}rg Albrecher.
\newblock {\em Ruin Probabilities}, volume~14 of {\em Advanced Series on
  Statistical Science and Applied Probability}.
\newblock World Scientific, 2nd edition, 2010.

\bibitem{asmussen2016orthonormal}
S{\o}ren Asmussen, Pierre-Olivier Goffard, and Patrick~J Laub.
\newblock Orthonormal polynomial expansions and lognormal sum densities.
\newblock In {\em Risk and Stochastics: Ragnar Norberg at 70}. Mathematical
  Finance Economics. World Scientific, 2018.

\bibitem{bowers1966expansion}
NL~Bowers.
\newblock Expansion of probability density functions as a sum of gamma
  densities with applications in risk theory.
\newblock {\em Transactions of Society of Actuaries}, 18(52):125--137, 1966.

\bibitem{CaTaWeZh08}
Jun Cai, Ken~Seng Tan, Chengguo Weng, and Yi~Zhang.
\newblock Optimal reinsurance under {VaR} and {CTE} risk measures.
\newblock {\em Insurance: Mathematics and Economics}, 43(1):185--196, 2008.

\bibitem{Ch10}
Ka~Chun Cheung.
\newblock Optimal reinsurance revisited: a geometric approach.
\newblock {\em ASTIN Bulletin}, 40(1):221--239, 005 2010.

\bibitem{ChTa11}
Yichun Chi and Ken~Seng Tan.
\newblock Optimal reinsurance under {VaR} and {CVaR} risk measures: a
  simplified approach.
\newblock {\em ASTIN Bulletin}, 41(2):487--509, 2011.

\bibitem{DeDhGoKa06}
Michel Denuit, Jan Dhaene, Marc~J. Goovaert, and Rob Kaas.
\newblock {\em Actuarial Theory for Dependent Risk: Measures, Orders and
  Models}.
\newblock John Wiley \& Sons, 2006.

\bibitem{DiZa91}
Persi Diaconis and Sandy Zabell.
\newblock Closed form summation for classical distributions: variations on a
  theme of de {M}oivre.
\newblock {\em Statistical Science}, 6(3):284--302, 1991.

\bibitem{DuGaMo09}
Daniel Dufresne, Jose Garrido, and Manuel Morales.
\newblock Fourier inversion formulas in option pricing and insurance.
\newblock {\em Methodology and Computing in Applied Probability},
  11(3):359--383, 2009.

\bibitem{embrechts2009panjer}
Paul Embrechts and Marco Frei.
\newblock Panjer recursion versus {FFT} for compound distributions.
\newblock {\em Mathematical Methods of Operations Research}, 69(3):497--508,
  2009.

\bibitem{StoplossCode}
Pierre-Olivier Goffard and Patrick~J. Laub.
\newblock {\em Online accompaniment for ``Orthogonal polynomial expansions to
  evaluate stop-loss premiums''}, 2017.
\newblock Available at
  \url{https://github.com/Pat-Laub/ActuarialOrthogonalPolynomials}.

\bibitem{GoLoPo15}
Pierre-Olivier Goffard, St{\'e}phane Loisel, and Denys Pommeret.
\newblock Polynomial approximations for bivariate aggregate claims amount
  probability distributions.
\newblock {\em Methodology and Computing in Applied Probability},
  19(1):151--174, 2015.

\bibitem{GoLoPo16}
Pierre-Olivier Goffard, St{\'e}phane Loisel, and Denys Pommeret.
\newblock A polynomial expansion to approximate the ultimate ruin probability
  in the compound {P}oisson ruin model.
\newblock {\em Journal of Computational and Applied Mathematics}, 296:499--511,
  2016.

\bibitem{GzNITa13}
Henryk Gzyl, Pier Luigi~Novi Inverardi, and Aldo Tagliani.
\newblock Determination of the probability of ultimate ruin by maximum entropy
  applied to fractional moments.
\newblock {\em Insurance: Mathematics and Economics}, 53(2):457--463, 2013.

\bibitem{GzTa12}
Henryk Gzyl and Aldo Tagliani.
\newblock Determination of the distribution of total loss from the fractional
  moments of its exponential.
\newblock {\em Applied Mathematics and Computation}, 219(4):2124--2133, 2012.

\bibitem{HaZa04}
Abdelhamid Hassairi and Mohammed Zarai.
\newblock Characterization of the cubic exponential families by orthogonality
  of polynomials.
\newblock {\em The Annals of Probability}, 32(3):2463--2476, 2004.

\bibitem{EIOPA}
European Insurance and Occupational~Pensions Authority.
\newblock {Q}uantitative impact studies {V}: Technical specifications.
\newblock Technical report, European Comission, Brussels, 2010.

\bibitem{JiPrRe16}
Tao Jin, Serge~B. Provost, and Jiandong Ren.
\newblock Moment-based density approximations for aggregate losses.
\newblock {\em Scandinavian Actuarial Journal}, 2016(3):216--245, 2016.

\bibitem{KaPrRe2019}
John Sang~Jin Kang, Serge~B. Provost, and Jiandong Ren.
\newblock Moment-based density approximation techniques as applied to
  heavy-tailed distributions.
\newblock {\em International Journal of Statistics and Probability}, 8(3),
  2019.

\bibitem{LeLi10}
Simon C.~K. Lee and X.~Sheldon Lin.
\newblock Modeling and evaluating insurance losses via mixtures of {E}rlang
  distributions.
\newblock {\em North American Actuarial Journal}, 14(1):107--130, 2010.

\bibitem{LePi11}
Claude Lef\`evre and Philippe Picard.
\newblock A new look at the homogeneous risk model.
\newblock {\em Insurance: Mathematics and Economics}, 49(3):512--519, 2011.

\bibitem{LeTrZu17}
Claude Lef\`evre, Julien Trufin, and Pierre Zuyderhoff.
\newblock Some comparison results for finite-time ruin probabilities in the
  classical risk model.
\newblock {\em Insurance: Mathematics and Economics}, 77(Supplement
  C):143--149, 2017.

\bibitem{LeMo90}
G{\'e}rard Letac and Marianne Mora.
\newblock Natural real exponential families with cubic variance functions.
\newblock {\em The Annals of Statistics}, 18(1):1--37, 1990.

\bibitem{Mnatsakanov2008}
Robert Mnatsakanov, LL~Ruymgaart, and Frits~H Ruymgaart.
\newblock Nonparametric estimation of ruin probabilities given a random sample
  of claims.
\newblock {\em Mathematical Methods of Statistics}, 17(1):35--43, 2008.

\bibitem{Mnatsakanov2013}
Robert~M. Mnatsakanov and Khachatur Sarkisian.
\newblock A note on recovering the distributions from exponential moments.
\newblock {\em Applied Mathematics and Computation}, 219(16):8730--8737, 2013.

\bibitem{Mnatsakanov2015}
Robert~M. Mnatsakanov, Khachatur Sarkisian, and A.~Hakobyan.
\newblock Approximation of the ruin probability using the scaled {L}aplace
  transform inversion.
\newblock {\em Applied Mathematics and Computation}, 268:717--727, 2015.

\bibitem{Mo82}
Carl~N. Morris.
\newblock Natural exponential families with quadratic variance functions.
\newblock {\em The Annals of Statistics}, 10(1):65--80, 1982.

\bibitem{NaChJi16}
Saralees Nadarajah, Jeffrey Chu, and Xiao Jiang.
\newblock On moment based density approximations for aggregate losses.
\newblock {\em Journal of Computational and Applied Mathematics}, 298:152--166,
  2016.

\bibitem{Ni96}
Ryuei Nishii.
\newblock {\em Orthogonal Functions of Inverse Gaussian Distributions}, pages
  243--250.
\newblock Springer, 1996.

\bibitem{PaWi81}
Harry~H. Panjer and Gordon~E. Willmot.
\newblock Finite sum evaluation of the negative binomial-exponential model.
\newblock {\em ASTIN Bulletin: The Journal of the IAA}, 12(2):133--137, 1981.

\bibitem{PaPaPo01}
Dmitry~E. Papush, Gary~S. Patrik, and Felix Podgaits.
\newblock Approximations of the aggregate loss distribution.
\newblock {\em CAS Forum (Winter)}, pages 175--186, 2001.

\bibitem{Pr05}
Serge~B. Provost.
\newblock Moment-based density approximants.
\newblock {\em Mathematica Journal}, 9(4):727--756, 2005.

\bibitem{RoScScTe08}
Tomasz Rolski, Hanspeter Schmidli, Volker Schmidt, and Jozef~L. Teugels.
\newblock {\em Stochastic Processes for Insurance and Finance}, volume 505 of
  {\em Wiley Series in Probability and Statistics}.
\newblock John Wiley \& Sons, 2009.

\bibitem{Sc12}
Wim Schoutens.
\newblock {\em Stochastic processes and orthogonal polynomials}, volume 146.
\newblock Springer Science \& Business Media, New York, 2012.

\bibitem{Szegoe1939}
Gabor Szeg{\"o}.
\newblock {\em Orthogonal Polynomials}, volume XXIII.
\newblock American Mathematical Society Colloquium Publications, 1939.

\bibitem{Na65}
B{\'e}la Sz{\"o}kefalvi-Nagy.
\newblock {\em Introduction to Real Functions and Orthogonal Expansions}.
\newblock Akad{\'e}miai Kiad{\'o}, 1965.

\bibitem{WiLi11}
Gordon~E. Willmot and X.~Sheldon Lin.
\newblock Risk modelling with the mixed {E}rlang distribution.
\newblock {\em Applied Stochastic Models in Business and Industry},
  27(1):2--16, 2011.

\bibitem{WiWo07}
Gordon~E. Willmot and Jae-Kyung Woo.
\newblock On the class of {E}rlang mixtures with risk theoretic applications.
\newblock {\em North American Actuarial Journal}, 11(2):99--115, 2007.

\end{thebibliography}
\end{document}